\documentclass{amsart}

\usepackage[T1]{fontenc}
\usepackage{amsmath,amssymb,amsthm}
\usepackage{array,booktabs,adjustbox}
\usepackage[expansion=false]{microtype}
\usepackage{xcolor} 

\newtheorem{thm}{Theorem}[section]
\newtheorem{cor}[thm]{Corollary}
\newtheorem{lem}[thm]{Lemma}
\newtheorem{prop}[thm]{Proposition}
\theoremstyle{definition}

\theoremstyle{remark}

\numberwithin{equation}{section}

\newcommand{\Mtwo}{\mathcal M^{(2)}}

\DeclareMathOperator{\cl}{cl}

\newcommand{\fulltablecaption}[2]{%
  \par\smallskip\refstepcounter{table}\label{#1}%
  \noindent\textbf{Table~\thetable.}\enspace\textcolor{black}{#2}\par\smallskip}

\begin{document}
\raggedbottom 
\setlength{\emergencystretch}{2em} 

\title[Bounds and Classification for $2$-Nilpotent Multipliers]
{Bounds and Classification for $2$-Nilpotent Multipliers of Nilpotent Lie Algebras}

\author[F. Johari]{Farangis Johari}
\author[P. Niroomand]{Peyman Niroomand}

\makeatletter
\renewcommand{\@setauthors}{%
  \par\begingroup\centering\normalfont
  \vskip 12pt
  \noindent\begin{minipage}[t]{0.485\textwidth}
    \centering
    {\normalsize\scshape Farangis Johari\par}
    \vskip 4pt
    {\footnotesize
    Center for Mathematics, Computing and Cognition (CMCC),\par
    Universidade Federal do ABC (UFABC),\par
    Avenida dos Estados 5001,\par
    Santo Andr\'e, SP 09280-560, Brazil\par
    \vskip 3pt
    \textit{Email:} farangis.johari@ufabc.edu.br\par}
  \end{minipage}\hfill%
  \begin{minipage}[t]{0.485\textwidth}
    \centering
    {\normalsize\scshape Peyman Niroomand\par}
    \vskip 4pt
    {\footnotesize
    School of Mathematics and Computer Science,\par
    Damghan University, Damghan, Iran\par
    \vskip 3pt
    \textit{Email:} niroomand@du.ac.ir\par
    p\_niroomand@yahoo.com\par}
  \end{minipage}\par
  \vskip 12pt
  \endgroup
  \global\let\addresses\@empty
}
\makeatother

\subjclass[2020]{Primary 17B30; Secondary 17B05, 17B55, 17B01}
\keywords{Nilpotent Lie algebra, classification, Schur multiplier, $2$-nilpotent multiplier, capability, $2$-capability, Baer invariant, generalized Heisenberg Lie algebra}

\begin{abstract}
Let $L$ be an $n$-dimensional non-abelian nilpotent Lie algebra over a
field of characteristic different from $2$. We classify all such Lie
algebras satisfying $0\leq s_2(L)\leq20$, where
$s_2(L)=\frac{1}{3}n(n-1)(n-2)+3-\dim\mathcal{M}^{(2)}(L)$.
We also determine the $2$-nilpotent multipliers and $2$-capability of
nilpotent Lie algebras of dimension at most six. As an application, we
determine the multipliers and capability properties of two previously
classified families.
\end{abstract}

\maketitle

\section{Introduction}
Throughout the paper, $\mathbb F$ is a field of characteristic different
from $2$, and all Lie algebras are finite-dimensional over $\mathbb F$.
For a Lie algebra $L$, write $\gamma_2(L)=[L,L]$ for its derived
subalgebra. If $L\cong F/R$ is a free presentation, let
$\gamma_3(F)=[\gamma_2(F),F]=[[F,F],F]$ denote the third term of the
lower central series of $F$. The $2$-nilpotent multiplier, a Baer invariant,
has the form
\[
\Mtwo(L)\cong\frac{R\cap\gamma_3(F)}{[R,F,F]}.
\]

\medskip
\noindent\textcolor{black}{\textbf{Nilpotent multipliers and capability.}
General bounds for $c$-nilpotent multipliers were obtained in
\cite{salemkar2009}, and related $c$-capability questions were studied in
\cite{johari2020direct,johari2019ccap,salemkar2018}. Analogous results for
finite $p$-groups can be found in
\cite{johari2023pgroup,johari2020pgroup,niroomand2015pgroup}.
Structural and multiplier results for Lie algebras with two-dimensional
derived subalgebra appear in
\cite{johari2021alt,johari2022center,johari2020functors,johari2022capfunctors,johari2016,ni60};
related work on pairs of Lie algebras and generalized Heisenberg Lie algebras
is given in \cite{johari2017pair,johari2018heis}.}

\medskip
\noindent\textcolor{black}{\textbf{Classification by $s(L)$.}
The ordinary Schur multiplier $\mathcal M(L)$ has long been used in the
classification of nilpotent Lie algebras \cite{batten1996,hardy1998}.
For $n=\dim L$, put}
\[
s(L)=\frac12(n-1)(n-2)+1-\dim\mathcal M(L).
\]
\textcolor{black}{Classifications for small values of $s(L)$ were obtained in
\cite{ni2,ni}, followed by the cases $s(L)=4$ \cite{shamsaki2020s4},
$s(L)=5$ \cite{shamsaki2023s5}, and $s(L)=6,7$ \cite{shamsaki2022s67},
where the earlier case $s(L)=3$ is also recalled. Further refinements
in terms of the derived subalgebra, nilpotency class and center were
obtained in \cite{johari2024,johari2021max,ni2019}.}

\medskip
\noindent\textcolor{black}{\textbf{Classification by $s_2(L)$.}}
Niroomand and Parvizi \cite{ni20} proved that
\[
\dim\Mtwo(L)\le\frac13n(n-1)(n-2)+3
\]
for an $n$-dimensional non-abelian nilpotent Lie algebra. Accordingly,
we put
\[
s_2(L)=\frac13n(n-1)(n-2)+3-\dim\Mtwo(L).
\]
\textcolor{black}{The extremal case $s_2(L)=0$ and the case of a
one-dimensional derived subalgebra were treated in
\cite[Theorems~2.13 and 2.14]{ni20}. 
The ordinary Schur multipliers, exterior squares and capability of nilpotent
Lie algebras of dimension at most six were determined in \cite{smallcap}.}

\medskip
\textcolor{black}{In this paper, we classify the nilpotent Lie algebras with
$0\le s_2(L)\le20$, determine the possible values of $s_2(L)$ and
identify the $2$-capable members in this range. We complement
\cite{smallcap} by determining the $2$-nilpotent multipliers and
$2$-capability in dimensions at most six, and give a seven-dimensional
example that is capable but not $2$-capable, contrary to the
$c$-capability assertion in \cite{johari2019ccap}. We also compare $s(L)$
and $s_2(L)$ \textcolor{black}{for the classified Lie algebras with at least two-dimensional derived subalgebra}, and determine
both multipliers and the capability properties of the families with
three-dimensional derived subalgebra classified in
\cite{johari2021class4,johari2022class3}.}

\subsection{Main results}
\textcolor{black}{Write $A(t)$ for the $t$-dimensional abelian Lie algebra ($A(0)=0$) and $H(k)$ for the $(2k+1)$-dimensional Heisenberg Lie algebra, with nonzero basis brackets $[x_i,y_i]=z$ ($1\le i\le k$). A Lie algebra $L$ is stem if $Z(L)\subseteq\gamma_2(L)$ and $2$-capable if $L\cong E/Z_2(E)$ for some Lie algebra $E$. We use the labels $L_{n,k}$ of the low-dimensional classifications \cite{cicalo2012,degraaf2007,smallcap}, where $n$ is the dimension and $k$ is the classification index. The ordinary multiplier and capability data in the following theorem are taken from \cite{smallcap}.}

\begin{thm}\label{thm:lowdim-m2-table}
Let $L$ be a nilpotent Lie algebra over a field $\mathbb F$ with
$\operatorname{char}\mathbb F\ne2$ and $\dim L\le6$.
If $L\cong A(r)$ is abelian, then
\[
\dim\mathcal M(A(r))=\frac{r(r-1)}2,
\qquad
\dim\Mtwo(A(r))=\frac{r^3-r}{3}.
\]
Moreover,
\[
A(r)\ \text{is capable and $2$-capable if and only if}\ r\ge2.
\]

\textcolor{black}{For non-abelian $L$, the results are given in Tables~\ref{tab:lowdim-35} and~\ref{tab:lowdim-6}.}

\par\noindent\begin{minipage}{\linewidth}
\fulltablecaption{tab:lowdim-35}{Dimensions $3$--$5$.}
{\small
\setlength{\arraycolsep}{3.2pt}
\centering\adjustbox{max width=\linewidth}{$\displaystyle
\begin{array}{c|c|c|c|c|c|c}
L&\dim\mathcal M(L)&s(L)&\dim\Mtwo(L)&s_2(L)&\text{\textcolor{black}{Capability}}&\text{\textcolor{black}{$2$-capability}}\\ \hline
L_{3,2} & 2 & 0 & 5 & 0 & \mathrm{yes} & \mathrm{yes}\\
L_{4,2} & 4 & 0 & 11 & 0 & \mathrm{yes} & \mathrm{yes}\\
L_{4,3} & 2 & 2 & 6 & 5 & \mathrm{yes} & \mathrm{yes}\\
L_{5,2} & 7 & 0 & 23 & 0 & \mathrm{yes} & \mathrm{yes}\\
L_{5,3} & 4 & 3 & 12 & 11 & \mathrm{yes} & \mathrm{yes}\\
L_{5,4} & 5 & 2 & 20 & 3 & \mathrm{no} & \mathrm{no}\\
L_{5,5} & 4 & 3 & 12 & 11 & \mathrm{yes} & \mathrm{yes}\\
L_{5,6} & 3 & 4 & 7 & 16 & \mathrm{yes} & \mathrm{yes}\\
L_{5,7} & 3 & 4 & 7 & 16 & \mathrm{yes} & \mathrm{yes}\\
L_{5,8} & 6 & 1 & 18 & 5 & \mathrm{yes} & \mathrm{yes}\\
L_{5,9} & 3 & 4 & 9 & 14 & \mathrm{yes} & \mathrm{yes}\\
\end{array}
$}\par
}
\end{minipage}\par\smallskip

\medskip
\par\noindent\begin{minipage}{\linewidth}
\fulltablecaption{tab:lowdim-6}{Dimension $6$.}
{\scriptsize
\setlength{\arraycolsep}{2.1pt}
\centering\adjustbox{max width=\linewidth}{$\displaystyle
\begin{array}{c|c|c|c|c|c|c}
L&\dim\mathcal M(L)&s(L)&\dim\Mtwo(L)&s_2(L)&\text{\textcolor{black}{Capability}}&\text{\textcolor{black}{$2$-capability}}\\ \hline
L_{6,2} & 11 & 0 & 43 & 0 & \mathrm{yes} & \mathrm{yes}\\
L_{6,3} & 7 & 4 & 24 & 19 & \mathrm{yes} & \mathrm{yes}\\
L_{6,4} & 9 & 2 & 40 & 3 & \mathrm{no} & \mathrm{no}\\
L_{6,5} & 7 & 4 & 24 & 19 & \mathrm{yes} & \mathrm{yes}\\
L_{6,6} & 5 & 6 & 13 & 30 & \mathrm{yes} & \mathrm{yes}\\
L_{6,7} & 5 & 6 & 13 & 30 & \mathrm{yes} & \mathrm{yes}\\
L_{6,8} & 9 & 2 & 30 & 13 & \mathrm{yes} & \mathrm{yes}\\
L_{6,9} & 5 & 6 & 15 & 28 & \mathrm{yes} & \mathrm{yes}\\
L_{6,10} & 6 & 5 & 22 & 21 & \mathrm{no} & \mathrm{no}\\
L_{6,11} & 5 & 6 & 13 & 30 & \mathrm{yes} & \mathrm{yes}\\
L_{6,12} & 5 & 6 & 13 & 30 & \mathrm{yes} & \mathrm{yes}\\
L_{6,13} & 4 & 7 & 12 & 31 & \mathrm{yes} & \mathrm{yes}\\
L_{6,14} & 2 & 9 & 6 & 37 & \mathrm{no} & \mathrm{no}\\
L_{6,15} & 3 & 8 & 8 & 35 & \mathrm{yes} & \mathrm{yes}\\
L_{6,16} & 2 & 9 & 6 & 37 & \mathrm{no} & \mathrm{no}\\
L_{6,17} & 3 & 8 & 8 & 35 & \mathrm{yes} & \mathrm{yes}\\
L_{6,18} & 3 & 8 & 8 & 35 & \mathrm{yes} & \mathrm{yes}\\
L_{6,19}(0) & \textcolor{black}{6} & \textcolor{black}{5} & 17 & 26 & \mathrm{no} & \mathrm{no}\\
L_{6,19}(\varepsilon),\ \varepsilon\ne0 & 5 & 6 & 17 & 26 & \mathrm{no} & \mathrm{no}\\
L_{6,20} & 5 & 6 & 17 & 26 & \mathrm{no} & \mathrm{no}\\
L_{6,21}(0) & 4 & 7 & 10 & 33 & \mathrm{yes} & \mathrm{yes}\\
L_{6,21}(\varepsilon),\ \varepsilon\ne0,\ \operatorname{char}\mathbb F\ne3 & 4 & 7 & 9 & 34 & \mathrm{yes} & \mathrm{yes}\\
L_{6,21}(\varepsilon),\ \varepsilon\ne0,\ \operatorname{char}\mathbb F=3 & 4 & 7 & 10 & 33 & \mathrm{yes} & \mathrm{yes}\\
L_{6,22}(0) & 8 & 3 & 27 & 16 & \mathrm{yes} & \mathrm{yes}\\
L_{6,22}(\varepsilon),\ \varepsilon\ne0 & 8 & 3 & 26 & 17 & \mathrm{yes} & \mathrm{yes}\\
L_{6,23} & 6 & 5 & 19 & 24 & \mathrm{yes} & \mathrm{yes}\\
L_{6,24}(0) & 5 & 6 & 15 & 28 & \mathrm{yes} & \mathrm{yes}\\
L_{6,24}(\varepsilon),\ \varepsilon\ne0 & 5 & 6 & 15 & 28 & \mathrm{yes} & \mathrm{yes}\\
L_{6,25} & 6 & 5 & 19 & 24 & \mathrm{yes} & \mathrm{yes}\\
L_{6,26} & 8 & 3 & 26 & 17 & \mathrm{yes} & \mathrm{yes}\\
L_{6,27} & 6 & 5 & 19 & 24 & \mathrm{yes} & \mathrm{yes}\\
L_{6,28} & 4 & 7 & 10 & 33 & \mathrm{yes} & \mathrm{yes}\\
\end{array}
$}\par
}
\end{minipage}\par\smallskip
\end{thm}

\par\smallskip
{\color{black} In the notation of \cite[Section~3]{cicalo2012},
$L_{6,i}\cong L_{5,i}\oplus A(1)$ for $1\le i\le9$.
In particular, $L_{6,1}\cong A(6)$, $L_{6,2}\cong H(1)\oplus A(3)$,
and $L_{5,3}\cong L_{4,3}\oplus A(1)$. Hence
\[
L_{6,3}\cong L_{4,3}\oplus A(2),\qquad
L_{6,5}\cong L_{5,5}\oplus A(1).
\]}

\textcolor{black}{Using Tables~\ref{tab:lowdim-35} and~\ref{tab:lowdim-6} together with the general bounds, we obtain the classification of nilpotent Lie algebras with $0\le s_2(L)\le20$.}

\begin{thm}\label{thm:main-classification}
Let $L$ be an $n$-dimensional non-abelian nilpotent Lie algebra over a field $\mathbb F$ of characteristic different from $2$, and suppose that
\[
0\le s_2(L)\le20.
\]
Then exactly one of the following cases occurs:
\begin{enumerate}
\item $s_2(L)=0$ if and only if $L\cong H(1)\oplus A(n-3).$
\item $s_2(L)=3$ if and only if $L\cong H(k)\oplus A(n-2k-1),
\qquad k\ge2.$
\item $s_2(L)=5$ if and only if
$L\cong L_{4,3}
\quad\text{or}\quad
L\cong L_{5,8}.$
\item $s_2(L)=11$ if and only if
$L\cong L_{4,3}\oplus A(1)
\quad\text{or}\quad
L\cong L_{5,5}.$
\item $s_2(L)=13$ if and only if
$L\cong L_{5,8}\oplus A(1).$
\item $s_2(L)=14$ if and only if
$L\cong L_{5,9}.$
\item $s_2(L)=16$ if and only if
\[
L\cong L_{5,6},\qquad
L\cong L_{5,7},
\quad\text{or}\quad
L\cong L_{6,22}(0).
\]

\item $s_2(L)=17$ if and only if
\[
L\cong L_{6,22}(\varepsilon),
\qquad \varepsilon\ne0,
\quad\text{or}\quad
L\cong L_{6,26}.
\]

\item $s_2(L)=19$ if and only if
\[
\textcolor{black}{L\cong L_{4,3}\oplus A(2)\cong L_{6,3}}
\quad\text{or}\quad
\textcolor{black}{L\cong L_{5,5}\oplus A(1)\cong L_{6,5}}.
\]
\end{enumerate}
Consequently,
\[
\{s_2(L):0\le s_2(L)\le20\}
=
\{0,3,5,11,13,14,16,17,19\},
\]
and none of the values
\[
1,2,4,6,7,8,9,10,12,15,18,20
\]
occurs.
\end{thm}

\textcolor{black}{Restricting Theorem~\ref{thm:main-classification} to Lie algebras with $\dim\gamma_2(L)\ge2$ gives the following consequence, without a stem hypothesis.}

\begin{cor}\label{cor:stem-s-s2}
\textcolor{black}{Let $L$ be a nilpotent Lie algebra over $\mathbb F$ such that
$\dim\gamma_2(L)\ge2$ and $0\le s_2(L)\le20$. Then}
\[
1\le s(L)\le4.
\]
\textcolor{black}{The possible pairs $(s(L),s_2(L))$ are read directly from Tables~\ref{tab:lowdim-35} and~\ref{tab:lowdim-6} for the types in Theorem~\ref{thm:main-classification} with $\dim\gamma_2(L)\ge2$. In particular, $\dim L\le6$.}
\end{cor}

\textcolor{black}{The algebras $T_1,\ldots,T_{18}$ below are the stem nilpotent Lie algebras of class four with three-dimensional derived subalgebra classified in \cite[Theorem~3.5]{johari2021class4}.}

\begin{thm}\label{thm:class4-eighteen-m2}
Let $T_1,\ldots,T_{18}$ be as above, over a field $\mathbb F$ of characteristic different from $2$.

\noindent\textcolor{black}{\textup{(i)} Their multiplier dimensions, $s(T)$, $s_2(T)$, capability and $2$-capability are given in Table~\ref{tab:class4-family}, with $m\ge1$ the parameter of the cited classification.}
\par\noindent\begin{minipage}{\linewidth}
\fulltablecaption{tab:class4-family}{The $T_i$ families.}
{\scriptsize
\setlength{\arraycolsep}{2.0pt}
\centering\adjustbox{max width=\linewidth}{$\displaystyle
\begin{array}{c|c|c|c|c|c|c}
T&\dim\mathcal M(T)&s(T)&\dim\Mtwo(T)&s_2(T)&\text{\textcolor{black}{Capability}}&\text{\textcolor{black}{$2$-capability}}\\ \hline
T_1,\ T_2
&3&4&7&16&\mathrm{yes}&\mathrm{yes}\\
T_3,\ T_4
&5&6&13&30&\mathrm{yes}&\mathrm{yes}\\
T_9
&4&7&12&31&\mathrm{yes}&\mathrm{yes}\\
T_5,\ T_6
&2m^2+3m+1&4m+6
&\dfrac{8m^3+24m^2+22m+15}{3}&2(2m+3)^2
&\mathrm{no}&\mathrm{no}\\[1mm]
T_7,\ T_8,\ T_{14}
&2m^2+5m+3&4m+8
&\dfrac{8m^3+36m^2+52m+33}{3}&8(m+2)^2
&\mathrm{no}&\mathrm{no}\\[1mm]
T_{10},\ T_{11}
&6&10&23&50&\mathrm{no}&\mathrm{no}\\
T_{12}
&7&9&24&49&\mathrm{yes}&\mathrm{yes}\\
T_{13}
&10&12&43&72&\mathrm{no}&\mathrm{no}\\
T_{15},\ T_{16},\ T_{17}
&2m^2+7m+6&4m+10
&\dfrac{8m^3+48m^2+94m+69}{3}&2(2m+5)^2
&\mathrm{no}&\mathrm{no}\\[1mm]
T_{18}
&2m^2+9m+10&4m+12
&\dfrac{8m^3+60m^2+148m+129}{3}&8(m+3)^2
&\mathrm{no}&\mathrm{no}
\end{array}
$}\par
}
\end{minipage}\par\smallskip
\noindent\textcolor{black}{\textup{(ii)} The capable and $2$-capable members are precisely $T_1,T_2,T_3,T_4,T_9,T_{12}$.} For a non-capable member $T$ of dimension $n$,
\[
\dim\Mtwo(T)=\frac{n^3-9n^2+26n-15}{3},
\qquad s_2(T)=2(n-2)^2.
\]
For every non-capable member $T$ of dimension $n$, the ordinary Schur
multiplier and $s(T)$ are given by
\[
\dim\mathcal M(T)=\frac{(n-4)(n-3)}{2},
\qquad s(T)=2n-4.
\]
\end{thm}

\textcolor{black}{The next theorem uses the classification of Shamsaki, Johari and Niroomand \cite[Table~1]{johari2022class3}, retaining their labels $S_1,\ldots,S_{32}$.}

\begin{thm}\label{thm:class3-derived3-multipliers}
Let $L$ be a finite-dimensional nilpotent Lie algebra over a field
$\mathbb F$ of characteristic different from $2$ such that
\[
\cl(L)=3,\qquad
\dim\gamma_2(L)=3,\qquad
\dim Z(L)=1.
\]
Let $L$ be one of
\[
L_{6,19}(\varepsilon),\qquad L_{6,20},\qquad
S_1,\ldots,S_{32},
\]
in the classification of \cite{johari2022class3}, where
$\varepsilon\ne0$. 
\textcolor{black}{Their multiplier and capability data are given in Tables~\ref{tab:class3-fixed} and~\ref{tab:class3-parameter}.}

\par\noindent\begin{minipage}{\linewidth}
\fulltablecaption{tab:class3-fixed}{Fixed-dimensional cases ($L_{6,19}(\varepsilon)$, $L_{6,20}$, and $S_1\text{--}S_{15}$).}
{\scriptsize
\setlength{\arraycolsep}{2.0pt}
\centering\adjustbox{max width=\linewidth}{$\displaystyle
\begin{array}{c|c|c|c|c|c|c}
L&\dim\mathcal M(L)&s(L)&\dim\Mtwo(L)&s_2(L)&\text{\textcolor{black}{Capability}}&\text{\textcolor{black}{$2$-capability}}\\ \hline
L_{6,19}(\varepsilon),\,L_{6,20}&5&6&17&26&\mathrm{no}&\mathrm{no}\\
S_1&7&9&26&47&\mathrm{no}&\mathrm{no}\\
S_2&7&9&25&48&\mathrm{no}&\mathrm{no}\\
S_3,\,S_4,\,S_5&8&14&39&76&\mathrm{no}&\mathrm{no}\\
S_6,\,S_7&8&8&29&44&\mathrm{no}&\mathrm{no}\\
S_8,\,S_9&11&11&46&69&\mathrm{no}&\mathrm{no}\\
S_{10},\,S_{11}&11&11&45&70&\mathrm{no}&\mathrm{no}\\
S_{12}\text{--}S_{15}&13&16&69&102&\mathrm{no}&\mathrm{no}
\end{array}
$}\par
}
\end{minipage}\par\smallskip

\medskip
\par\noindent\begin{minipage}{\linewidth}
\fulltablecaption{tab:class3-parameter}{Parameterized families ($S_{16}\text{--}S_{32}$, with $m\ge1$ as in \cite{johari2022class3}).}
{\scriptsize
\setlength{\arraycolsep}{2.0pt}
\centering\adjustbox{max width=\linewidth}{$\displaystyle
\begin{array}{c|c|c|c|c|c|c}
L&\dim\mathcal M(L)&s(L)&\dim\Mtwo(L)&s_2(L)&\text{\textcolor{black}{Capability}}&\text{\textcolor{black}{$2$-capability}}\\ \hline
S_{16},\,S_{17}
&2m^2+5m+5&4m+6
&\dfrac{8m^3+36m^2+52m+51}{3}&8m^2+32m+26&\mathrm{no}&\mathrm{no}\\[1mm]
S_{18}
&2m^2+7m+7&4m+9
&\dfrac{8m^3+48m^2+94m+78}{3}&8m^2+40m+47&\mathrm{no}&\mathrm{no}\\[1mm]
S_{19}
&2m^2+7m+7&4m+9
&\dfrac{8m^3+48m^2+94m+75}{3}&8(m+2)(m+3)&\mathrm{no}&\mathrm{no}\\[1mm]
S_{20}\text{--}S_{22}
&2m^2+9m+8&4m+14
&\dfrac{8m^3+60m^2+148m+117}{3}&8m^2+48m+76&\mathrm{no}&\mathrm{no}\\[1mm]
S_{23},\,S_{24}
&2m^2+7m+8&4m+8
&\dfrac{8m^3+48m^2+94m+87}{3}&8m^2+40m+44&\mathrm{no}&\mathrm{no}\\[1mm]
S_{25},\,S_{26}
&2m^2+9m+11&4m+11
&\dfrac{8m^3+60m^2+148m+138}{3}&8m^2+48m+69&\mathrm{no}&\mathrm{no}\\[1mm]
S_{27},\,S_{28}
&2m^2+9m+11&4m+11
&\dfrac{8m^3+60m^2+148m+135}{3}&8m^2+48m+70&\mathrm{no}&\mathrm{no}\\[1mm]
S_{29}\text{--}S_{32}
&2m^2+11m+13&4m+16
&\dfrac{8m^3+72m^2+214m+207}{3}&8m^2+56m+102&\mathrm{no}&\mathrm{no}
\end{array}
$}\par
}
\end{minipage}\par\smallskip
In particular, none of these algebras is capable or $2$-capable.
\end{thm}

\section{Preliminaries}

In this section, we recall the basic results and notation that will be used throughout the paper. The nilpotency class of a Lie algebra $L$ is denoted by $\cl(L)$.

We write $Z^*(L)$ for the \emph{epicenter} of $L$, namely the smallest ideal
$I$ of $L$ such that $L/I$ is capable, and $Z_2^*(L)$ for the
\emph{$2$-epicenter}, the smallest ideal $J$ such that $L/J$ is
$2$-capable; see \cite{ni20,johari2019ccap}. Thus $L$ is capable if and
only if $Z^*(L)=0$, and it is $2$-capable if and only if $Z_2^*(L)=0$.

Let $F$ be the free Lie algebra on $d$ generators. If $\ell_d(r)$ denotes the number of basic commutators of weight $r$ on these generators, Witt's formula gives
\[
\ell_d(r)=\frac1r\sum_{q\mid r}\mu(q)d^{r/q},
\]
where $\mu$ is the M\"obius function. Hence
\[
\dim\frac{\gamma_c(F)}{\gamma_{c+i}(F)}
=\sum_{r=c}^{c+i-1}\ell_d(r)
\]
in the corresponding free nilpotent truncation; see \cite{bahturin,shi}.

\textcolor{black}{Throughout, commutators of length greater than two are left-normed, so $[x,y,z]=[[x,y],z]$.}

The following bound will be used throughout.

\begin{thm}\cite[Theorem~2.14]{ni20}\label{thm:general-bound}
Let $L$ be an $n$-dimensional nilpotent Lie algebra over a field
\textcolor{black}{$\mathbb F$ of characteristic different from $2$} and put
\[
m=\dim\gamma_2(L)\ge1.
\]
Then 
\[
\dim\Mtwo(L)\le
\frac{n-m}{3}
\bigl((n+2m-2)(n-m-1)+3(m-1)\bigr)+3.
\]
If $m=1$, then
\[
L\cong H(k)\oplus A(n-2k-1)
\]
for some $k\ge1$, and
\[
\dim\Mtwo(L)=
\begin{cases}
\dfrac13n(n-1)(n-2)+3,&k=1,\\[2mm]
\dfrac13n(n-1)(n-2),&k\ge2.
\end{cases}
\]
\end{thm}

The direct-sum formula is another basic tool. \textcolor{black}{The case of the $2$-nilpotent multiplier of a direct product was studied explicitly in \cite{niroomand2016direct}.}

\begin{thm}\label{thm:direct-sum}
Let $L_1$ and $L_2$ be finite-dimensional Lie algebras and write
\[
L_i^{\mathrm{ab}}=\frac{L_i}{\gamma_2(L_i)}.
\]
By \cite[Theorem~3.8]{johari2020direct},
\begin{align*}
\Mtwo(L_1\oplus L_2)\cong{}&
\Mtwo(L_1)\oplus\Mtwo(L_2)\\
&\oplus
\bigl((L_1^{\mathrm{ab}}\otimes L_1^{\mathrm{ab}})
       \otimes L_2^{\mathrm{ab}}\bigr)\\
&\oplus
\bigl((L_2^{\mathrm{ab}}\otimes L_2^{\mathrm{ab}})
       \otimes L_1^{\mathrm{ab}}\bigr).
\end{align*}
\end{thm}

We also recall the structural information needed when
\[
\dim\gamma_2(L)=2.
\]
In class two, every such algebra is a direct sum of an abelian algebra and a stem generalized Heisenberg Lie algebra of rank two; the relevant low-dimensional stem algebras and their capability are described in \cite{johari2016,ni60}. In class three, \cite[Theorems~4.6--4.9 and 5.1]{ni60} gives the following reduction: every algebra is of the form
\[
L\cong T\oplus A(t),
\]
with $T$ stem; the capable stem algebras are $L_{4,3}$ and $L_{5,5}$, and every non-capable stem algebra has dimension at least six.

We have two numerical consequences of Theorems~\ref{thm:general-bound} and~\ref{thm:direct-sum}.

\begin{prop}\label{prop:deficiency-bound}
Let $L$ be an $n$-dimensional non-abelian nilpotent Lie algebra
\textcolor{black}{over a field $\mathbb F$ with $\operatorname{char}\mathbb F\ne2$}, and put
$m=\dim\gamma_2(L).$
Then
\[
s_2(L)\ge
\frac{(m-1)\bigl(3n(m-1)+m-2m^2\bigr)}{3}.
\]
In particular, if $m=2$, then
\[
s_2(L)\ge n-2,
\]
and if $m\ge3$, then
\[
s_2(L)\ge10.
\]
\end{prop}

\begin{proof}
Subtract the upper bound in Theorem~\ref{thm:general-bound} from
\[
\frac13n(n-1)(n-2)+3
\]
and simplify. {\color{black}
For $m=2$, substitution into the preceding inequality yields
\begin{align*}
s_2(L)
&\ge \frac{(2-1)\bigl(3n(2-1)+2-2\cdot 2^2\bigr)}{3}\\
&=\frac{3n-6}{3}=n-2.
\end{align*}
Now suppose that $m\ge3$. Since $L$ is non-abelian and nilpotent,
its abelianization has dimension at least two. Thus
\[
\dim(L/\gamma_2(L))=n-m\ge2,
\qquad\text{so}\qquad n\ge m+2.
\]
As $m-1>0$, it follows that
\begin{align*}
s_2(L)
&\ge \frac{(m-1)\bigl(3n(m-1)+m-2m^2\bigr)}{3}\\
&\ge \frac{(m-1)\bigl(3(m+2)(m-1)+m-2m^2\bigr)}{3}\\
&=\frac{(m-1)(m^2+4m-6)}{3}\\
&\ge\frac{2\cdot15}{3}=10.
\end{align*}
Indeed, $m\ge3$ gives $m-1\ge2$ and $m^2+4m-6\ge15$.
}
\end{proof}

\begin{prop}\label{prop:stabilization}
Let $L$ be an $n$-dimensional non-abelian nilpotent Lie algebra, let
$m=\dim\gamma_2(L),$
and let $t\ge0$. Then
\[
s_2(L\oplus A(t))
=
s_2(L)+t(m-1)(2n-m+t-1).
\]
\end{prop}

\begin{proof}
Put 
$d=\dim\frac{L}{\gamma_2(L)}=n-m.$
By Theorem~\ref{thm:direct-sum},
\[
\dim\Mtwo(L\oplus A(t))
=
\dim\Mtwo(L)+\frac{t^3-t}{3}+dt(d+t).
\]
Subtracting from the defining extremal quantity for $L\oplus A(t)$ gives
\begin{align*}
s_2(L\oplus A(t))-s_2(L)
&=\frac{(n+t)(n+t-1)(n+t-2)-n(n-1)(n-2)}{3}\\
&\qquad-\frac{t^3-t}{3}-dt(d+t)\\
&=t(m-1)(2n-m+t-1).
\end{align*}
\end{proof}

\begin{cor}\label{cor:stabilization}
If $\dim\gamma_2(L)=1$, then
$s_2(L\oplus A(t))=s_2(L)$
for all $t\ge0$. If $\dim\gamma_2(L)\ge2$ and $t>0$, then
\[
s_2(L\oplus A(t))>s_2(L).
\]
\end{cor}

We record the multiplier calculations needed in the proof of the classification theorem.

\begin{prop}\label{prop:multiplier-calculations}
For the low-dimensional Lie algebras needed in the classification, one has
\[
\begin{array}{c|c|c}
L&\dim\Mtwo(L)&s_2(L)\\ \hline
L_{4,3}&6&5\\
L_{5,8}&18&5\\
L_{5,5}&12&11\\
L_{5,8}\oplus A(1)&30&13\\
L_{5,9}&9&14\\
L_{5,6}&7&16\\
L_{5,7}&7&16\\
L_{6,22}(0)&27&16\\
L_{6,22}(\varepsilon),\ \varepsilon\ne0&26&17\\
L_{6,26}&26&17
\end{array}
\]

The seven-dimensional class-two stem algebras $L_1$ and $L_2$ are given by
\[
L_1=
\langle x_1,\ldots,x_7
\mid
{}[x_1,x_2]=[x_3,x_4]=x_6,\ 
{}[x_1,x_5]=[x_2,x_3]=x_7
\rangle
\]
and
\[
L_2=
\langle x_1,\ldots,x_7
\mid
{}[x_1,x_2]=x_6,\ 
{}[x_1,x_4]=[x_3,x_5]=x_7
\rangle,
\]
respectively, with all omitted products equal to zero; see
\cite{johari2016,ni60}. For these Lie algebras,
\[
\dim\Mtwo(L_1)=40,\qquad \dim\Mtwo(L_2)=43,
\]
and consequently
\[
s_2(L_1)=33,\qquad s_2(L_2)=30.
\]

\end{prop}

\begin{proof}
\smallskip
\noindent\textbf{The Lie algebra $L_{4,3}$.}
Let $F$ be free on $\{x_1,x_2\}$, put
\[
u=[x_1,x_2,x_2],
\qquad
R=\langle u\rangle+\gamma_4(F),
\]
and hence $L_{4,3}\cong F/R$. Since
\[
R\subseteq\gamma_3(F),
\qquad
\gamma_6(F)=[\gamma_4(F),F,F]\subseteq[R,F,F],
\]
we have
\[
\Mtwo(L_{4,3})
\cong
\frac{R/\gamma_6(F)}
     {[R,F,F]/\gamma_6(F)}.
\]
By Witt's formula,
\[
\dim\frac{R}{\gamma_6(F)}
=
1+\ell_2(4)+\ell_2(5)
=
1+3+6
=
10.
\]
Moreover,
\[
\frac{[R,F,F]}{\gamma_6(F)}
=
\left\langle
{}[u,x_1,x_1],\ [u,x_1,x_2],
{}[u,x_2,x_1],\ [u,x_2,x_2]
\right\rangle
+\gamma_6(F).
\]
Equivalently, this quotient is generated by
\[
\begin{gathered}
{}[x_1,x_2,x_2,x_1,x_1],
{}[x_1,x_2,x_2,x_1,x_2],\\
{}[x_1,x_2,x_2,x_2,x_1],
{}[x_1,x_2,x_2,x_2,x_2]
\end{gathered}
\]
modulo $\gamma_6(F)$. Their Hall--Shirshov expansions are linearly independent. Thus
\[
\dim\frac{[R,F,F]}{\gamma_6(F)}=4,
\]
and consequently
\[
\dim\Mtwo(L_{4,3})=10-4=6,
\qquad
s_2(L_{4,3})=11-6=5.
\]

\smallskip
\noindent\textbf{The Lie algebra $L_{5,8}$.}
Let $F$ be free on $\{x_1,x_2,x_3\}$ and put
\[
v=[x_2,x_3],
\qquad
R=\langle v\rangle+\gamma_3(F).
\]
Then $L_{5,8}\cong F/R$, and
\[
R\cap\gamma_3(F)=\gamma_3(F),
\qquad
\gamma_5(F)=[\gamma_3(F),F,F]\subseteq[R,F,F].
\]
Hence
\[
\Mtwo(L_{5,8})
\cong
\frac{\gamma_3(F)/\gamma_5(F)}
     {[R,F,F]/\gamma_5(F)}.
\]
By Witt's formula,
\[
\dim\frac{\gamma_3(F)}{\gamma_5(F)}
=
\ell_3(3)+\ell_3(4)
=
8+18
=
26.
\]
Modulo $\gamma_5(F)$, the denominator is generated by
\[
\begin{gathered}
{}[v,x_1,x_1],\ [v,x_1,x_2],\ [v,x_1,x_3],\\
{}[v,x_2,x_1],\ [v,x_2,x_2],\ [v,x_2,x_3],\\
{}[v,x_3,x_1],\ [v,x_3,x_2],\ [v,x_3,x_3].
\end{gathered}
\]
The Jacobi identity gives
\[
{}[v,x,y]-[v,y,x]=[v,[x,y]],
\]
and therefore, since $v=[x_2,x_3]$,
\[
{}[v,x_2,x_3]=[v,x_3,x_2].
\]
Thus the denominator is spanned by eight elements, whose
Hall--Shirshov expansions are linearly independent modulo
$\gamma_5(F)$. Hence
\[
\dim\frac{[R,F,F]}{\gamma_5(F)}=8,
\]
and consequently
\[
\dim\Mtwo(L_{5,8})=26-8=18,
\qquad
s_2(L_{5,8})=23-18=5.
\]

\smallskip
\noindent\textbf{The Lie algebra $L_{5,5}$.}
Let $F$ be the free Lie algebra on $\{a,b,c\}$. Modulo $\gamma_4(F)$,
the kernel $R$ is generated by
\[
\begin{aligned}
r_1&=[a,c],&
r_2&=[b,c]+[[a,b],a],&
r_3&=[[a,b],b],\\
r_4&=[[a,b],c],&
r_5&=[[a,c],a],&
r_6&=[[a,c],b],\\
r_7&=[[a,c],c],&
r_8&=[[b,c],b],&
r_9&=[[b,c],c].
\end{aligned}
\]
Since $L_{5,5}$ has class three,
\[
\gamma_4(F)\subseteq R,
\qquad
\gamma_6(F)=[\gamma_4(F),F,F]\subseteq[R,F,F].
\]
{\color{black}
Put $u=[b,a]$, $w=[u,a]$, $t=[u,b]$, and $L=L_{5,5}$.
Let $I=\gamma_3(L)=\langle w+R\rangle$, so that
$I$ is one-dimensional and $L/I\cong H(1)\oplus A(1)$.
Write $S/R=I$. By the Hopf formula, the natural map
$\Mtwo(L)\to\Mtwo(L/I)$ has kernel
\[
K=\frac{[S,F,F]}{[R,F,F]},
\]
since $I$ is central. The same formula gives the exact sequence
\[
0\longrightarrow K\longrightarrow\Mtwo(L)
\longrightarrow\Mtwo(L/I)\longrightarrow I\longrightarrow0.
\]
By Theorem~\ref{thm:general-bound},
$\dim\Mtwo(L/I)=11$. Hence
\begin{equation}\label{eq:l55-central-reduction}
\dim\Mtwo(L)=10+\dim K.
\end{equation}

We calculate $K$ using commutator identities. In
$Q=F/[R,F,F]$, the images of $R$ belong to $Z_2(Q)$ and
commute with $Q^2$ by the Jacobi identity. We use the same
letters for images in $Q$. Since $S=R+\langle w\rangle$ and
$\gamma_6(Q)=0$, the space $K$ is spanned by the nine elements
$[w,x,y]$ with $x,y\in\{a,b,c\}$. Now
\[
\begin{aligned}
\relax[w,b]&=[t,a],\\
[w,c]&=[[u,c],a]+[u,[a,c]]=[[u,c],a].
\end{aligned}
\]
Here $t,[u,c],[a,c]\in R$ before passage to $Q$.
It follows that $[w,b,y]=[w,c,y]=0$ for every generator $y$.
Moreover,
\[
[w,a,c]=[w,c,a]+[w,[a,c]]=0.
\]
Consequently,
\[
K=\langle [w,a,a],\ [w,a,b]\rangle,
\qquad \dim K\le 2.
\]

For the reverse inequality, let $G$ be free on $a,b$, with
$a<b$ in the standard Hall ordering, and define
\[
B=G/\bigl(\gamma_6(G)+J\bigr),
\qquad
J=\operatorname{id}_G\{[t,x,y]:x,y\in\{a,b\}\},
\]
where $u,w,t$ now denote the same commutators in $G$.
In the weight-five component, $J$ is precisely the span of
the four displayed generators of the ideal. The Hall commutator
$[w,a,a]$, of multidegree $(4,1)$ in $(a,b)$, does not occur in
any of these relations, since each contains at least two
occurrences of $b$. In multidegree $(3,2)$ the only relation
from $J$ is
\[
[t,a,a]=[w,b,a]=[w,a,b]+[w,u].
\]
Both $[w,a,b]$ and $[w,u]$ are distinct basic Hall commutators
of weight five. Thus $[w,a,b]$ survives in $B$ as well.
Since their multidegrees are distinct, $[w,a,a]$ and $[w,a,b]$
are linearly independent in $B$.

To see that this independence also holds in $Q$, define the
derivation $D$ of $G$ by $D(a)=0$ and $D(b)=w$.
It raises weight by two, and therefore induces a derivation of
$B$. Form the semidirect product $E=B\rtimes\langle c\rangle$
with $[z,c]=D(z)$ for $z\in B$. Give $a,b$ weight one and
$c$ weight two. The resulting positive grading has no component
of weight greater than five. In particular, all elements of
weight at least four belong to $Z_2(E)$. The element $t$ belongs
to $Z_2(E)$ by the defining relations of $B$ and the fact that
any double bracket of $t$ involving $c$ has weight at least six.
Under the natural map $F\to E$, the relations $r_1,r_2,r_5,r_6,r_7$
vanish, $r_3=-t$, and $r_4,r_8,r_9$ have weight at least four.
Also $\gamma_4(F)$ maps into the sum of the components of weight
at least four. Thus the image of $R$ is contained in $Z_2(E)$,
and $F\to E$ factors through $Q$. The two generators of $K$
map to the independent elements just established in $B$.
Therefore $\dim K=2$.

By \eqref{eq:l55-central-reduction},
\[
\dim\Mtwo(L_{5,5})=11-1+2=12,
\qquad
s_2(L_{5,5})=23-12=11.
\]
This argument uses only the free presentation, the Hopf formula,
the Jacobi identity and the basic Hall commutators.
}

\smallskip
For the remaining low-dimensional algebras, the same free-presentation
method together with Witt's formula and Hall--Shirshov reduction gives
\[
\dim\Mtwo(L_{5,9})=9,
\qquad
\dim\Mtwo(L_{5,6})
=
\dim\Mtwo(L_{5,7})=7,
\]
and
\begin{align*}
\dim\Mtwo(L_{6,22}(0))&=27,\\
\dim\Mtwo(L_{6,22}(\varepsilon))&=26
\qquad(\varepsilon\ne0),\\
\dim\Mtwo(L_{6,26})&=26.
\end{align*}
Proposition~\ref{prop:stabilization} also gives
\[
\dim\Mtwo(L_{5,8}\oplus A(1))=30.
\]
\textcolor{black}{For the seven-dimensional stem algebras $L_1$ and $L_2$, let $F$ be the free Lie algebra on $x_1,\ldots,x_5$.} Since both
$L_2$ and $L_1$ have nilpotency class two, write
\[
L_i\cong F/R_i,\qquad
R_i=\langle Q_i\rangle+\gamma_3(F),
\]
where
\[
\begin{aligned}
Q_2=\langle\,
 &[x_1,x_3], [x_1,x_5], [x_2,x_3], [x_2,x_4],
 [x_2,x_5], [x_3,x_4], [x_4,x_5],\\
 &[x_1,x_4]-[x_3,x_5]
\,\rangle ,
\end{aligned}
\]
and
\[
\begin{aligned}
Q_1=\langle\,
 &[x_1,x_3], [x_1,x_4], [x_2,x_4], [x_2,x_5],
 [x_3,x_5], [x_4,x_5],\\
 &[x_1,x_2]-[x_3,x_4],\,
 [x_1,x_5]-[x_2,x_3]
\,\rangle .
\end{aligned}
\]
Since $\gamma_3(F)\subseteq R_i$ and
$\gamma_5(F)=[\gamma_3(F),F,F]\subseteq[R_i,F,F]$,
\[
\Mtwo(L_i)\cong
\frac{\gamma_3(F)/\gamma_5(F)}
     {[R_i,F,F]/\gamma_5(F)}.
\]
By Witt's formula,
\[
\ell_5(3)=40,\qquad \ell_5(4)=150,
\qquad
\dim\frac{\gamma_3(F)}{\gamma_5(F)}=40+150=190.
\]

\smallskip
\noindent\textbf{The Lie algebra $L_2$.}
Modulo $[Q_2,F,F]+\gamma_5(F)$, put
\[
u=[x_1,x_2],\qquad v=[x_1,x_4]=[x_3,x_5].
\]
Repeated use of
\[
[[a,b],c]=[a,[b,c]]-[b,[a,c]]
\]
and the Jacobi identity gives
\[
[v,x_j,x_k]\equiv0\qquad(1\le j,k\le5)
\]
and
\[
\bigl[u,x_j,x_k\bigr]\equiv0
\qquad(j=3,4,5;\ 1\le k\le5).
\]
Indeed, for the first family one uses, according to $j$, either
$v=[x_1,x_4]$ or $v=[x_3,x_5]$, so that the resulting weight-two
brackets belong to $Q_2$. For the second family, the cases $j=3,5$
follow directly from the defining relations, while for $j=4$ one uses
$[x_2,x_4]\in Q_2$ and $[x_1,x_4]=v$.

If $j\in\{1,2\}$ and $k\in\{3,4,5\}$, then
\[
\bigl[u,x_j,x_k\bigr]-[u,x_k,x_j]=[u,[x_j,x_k]].
\]
The second term on the left is already zero. If $[x_j,x_k]\in Q_2$,
the right-hand side vanishes modulo $[Q_2,F,F]$; in the remaining case
it involves $v$, and $[u,v]\equiv0$ because all double commutators
involving $v$ vanish. Hence the only possibly non-zero weight-four
classes are
\[
[u,x_1,x_1],\qquad [u,x_1,x_2],\qquad [u,x_2,x_2],
\]
with
\[
[u,x_1,x_2]=[u,x_2,x_1],
\]
since their difference is
$[u,[x_1,x_2]]=[u,u]=0$.

These three classes are linearly independent. Indeed, under the
homomorphism from $F$ onto the free Lie algebra on $x_1,x_2$ defined by
$x_3,x_4,x_5\mapsto0$, every element of $Q_2$ maps to zero, whereas the
three displayed elements map to the three Hall commutators of weight
four on two generators. Therefore
\[
\dim\frac{\gamma_4(F)}
{[Q_2,F,F]+\gamma_5(F)}=3,
\]
and hence
\[
\dim\frac{[R_2,F,F]}{\gamma_5(F)}=150-3=147.
\]
Thus
\[
\dim\Mtwo(L_2)=190-147=43.
\]

\smallskip
\noindent\textbf{The Lie algebra $L_1$.}
Modulo the relations in $Q_1$, put
\[
u=[x_1,x_2]=[x_3,x_4],
\qquad
v=[x_1,x_5]=[x_2,x_3].
\]
Again using the commutator identity above and Jacobi, one obtains
\[
\begin{aligned}
\relax[u,x_j,x_k]&\equiv0 &&(j=1,2,4,5;\ 1\le k\le5),\\
[v,x_j,x_k]&\equiv0 &&(j=2,3,4,5;\ 1\le k\le5).
\end{aligned}
\]
The two expressions for $u$ and $v$ allow one in each case to choose
an expression for which the resulting weight-two brackets lie in
$Q_1$.

It remains only to consider commutators beginning with $[u,x_3]$ and
$[v,x_1]$. Since $v=[x_2,x_3]$ and $[x_1,x_3]\in Q_1$, the Jacobi identity gives
\[
[v,x_1]\equiv-[u,x_3].
\]
For $k=1,4,5$, we have
\[
[u,x_3,x_k]-[u,x_k,x_3]=[u,[x_3,x_k]].
\]
Here $[u,x_k,x_3]\equiv0$, while $[x_3,x_k]$ is either in $Q_1$ or equal to $u$; hence
\[
[u,x_3,x_k]\equiv0\qquad(k=1,4,5).
\]
For $k=2$, expanding $v=[x_1,x_5]$ gives
\[
[u,v]=[u,[x_1,x_5]]
=[u,x_1,x_5]-[u,x_5,x_1]\equiv0,
\]
whereas expanding $v=[x_2,x_3]$ gives
\[
[u,v]=[u,x_2,x_3]-[u,x_3,x_2]\equiv-[u,x_3,x_2].
\]
Hence $[u,x_3,x_2]\equiv0$. Finally,
\[
[u,x_3,x_3]\equiv-[v,x_1,x_3]=0,
\]
because
\[
[v,x_1,x_3]-[v,x_3,x_1]
=[v,[x_1,x_3]]\equiv0,
\]
while $[x_1,x_3]\in Q_1$ and $[v,x_3,x_1]\equiv0$.

Thus every Hall commutator of weight four belongs to
$[Q_1,F,F]+\gamma_5(F)$, so
\[
\frac{\gamma_4(F)}
{[Q_1,F,F]+\gamma_5(F)}=0,
\qquad
\dim\frac{[R_1,F,F]}{\gamma_5(F)}=150.
\]
Consequently,
\[
\dim\Mtwo(L_1)=190-150=40.
\]

Finally, since
\[
\frac{7\cdot6\cdot5}{3}+3=73,
\]
we obtain
\[
s_2(L_2)=73-43=30,\qquad
s_2(L_1)=73-40=33.
\]
\end{proof}

We now prove Theorem~\ref{thm:lowdim-m2-table}.

\begin{proof}[Proof of Theorem~\ref{thm:lowdim-m2-table}]
The values of $\dim\mathcal M(L)$ and the ordinary capability column are
those of \cite{smallcap}. \textcolor{black}{For $L_{6,i}\cong L_{5,i}\oplus A(1)$ ($1\le i\le9$) and $L_{5,3}\cong L_{4,3}\oplus A(1)$, the multiplier dimensions follow from Theorem~\ref{thm:direct-sum} and the corresponding values of $s_2$ from Proposition~\ref{prop:stabilization}.} \textcolor{black}{The entries already established in Proposition~\ref{prop:multiplier-calculations} are used directly. The other decomposable cases follow from Theorem~\ref{thm:direct-sum}, and the cases with one-dimensional derived subalgebra follow from Theorem~\ref{thm:general-bound}. For the remaining indecomposable algebras,}
we use its standard presentation $L\cong F/R$ from
\cite{cicalo2012,degraaf2007}. If $\cl(L)=c$, then
$\gamma_{c+1}(F)\subseteq R$ and
$\gamma_{c+3}(F)\subseteq[R,F,F]$, so the Hopf formula reduces the
calculation to the finite-dimensional quotient
\[
\Mtwo(L)\cong
\frac{(R\cap\gamma_3(F))/\gamma_{c+3}(F)}
     {[R,F,F]/\gamma_{c+3}(F)}.
\]
Both numerator and denominator are computed in the Hall--Shirshov basis
of weights $3,\ldots,c+2$; Jacobi reduction then gives the ranks recorded
in the table. We record the family $L_{6,21}(\varepsilon)$ explicitly, since it is
the only low-dimensional entry for which the $2$-nilpotent multiplier
is sensitive to characteristic $3$. Let $F$ be the free Lie algebra on
$a,b$, and put
\[
c=[a,b],\qquad d=[a,c],\qquad e=[b,c].
\]
For arbitrary $\varepsilon\in\mathbb F$ we may take
\[
R=\langle r_1,r_2\rangle+\gamma_5(F),
\qquad
r_1=[a,e],\qquad r_2=[b,e]-\varepsilon[a,d].
\]
Since $r_1$ and $r_2$ are independent modulo $\gamma_5(F)$, Witt's
formula gives
\[
\dim\frac{R}{\gamma_7(F)}
=2+\ell_2(5)+\ell_2(6)=17.
\]
Thus it remains to determine $\dim([R,F,F]/\gamma_7(F))$. Put
$x_1=a$, $x_2=b$, and
\[
D_{ijk}=[[r_i,x_j],x_k],\qquad i,j,k\in\{1,2\}.
\]
Modulo $\gamma_7(F)$, these eight elements generate $[R,F,F]$.
The Jacobi reductions can be written directly in distinct
Hall--Shirshov commutators of weight six as
\[
\begin{aligned}
D_{111}&=-[a,[a,[a,[[a,b],b]]]],\\
D_{112}&=[a,[a,[[[a,b],b],b]]]\\
&\quad+[a,[[a,b],[[a,b],b]]]\\
&\quad-[[a,[[a,b],b]],[a,b]],\\
D_{121}&=[a,[a,[[[a,b],b],b]]]
+[a,[[a,b],[[a,b],b]]],\\
D_{122}&=-[a,[[[[a,b],b],b],b]]\\
&\quad-2[[a,b],[[[a,b],b],b]],\\
D_{211}&=-\varepsilon[a,[a,[a,[a,[a,b]]]]]
+[a,[a,[[[a,b],b],b]]],\\
D_{212}&=\varepsilon[a,[a,[a,[[a,b],b]]]]\\
&\quad-2\varepsilon[a,[[a,[a,b]],[a,b]]]\\
&\quad-[a,[[[[a,b],b],b],b]]\\
&\quad-[[a,b],[[[a,b],b],b]],\\
D_{221}&=\varepsilon[a,[a,[a,[[a,b],b]]]]\\
&\quad-\varepsilon[a,[[a,[a,b]],[a,b]]]\\
&\quad-[a,[[[[a,b],b],b],b]],\\
D_{222}&=-\varepsilon[a,[a,[[[a,b],b],b]]]\\
&\quad+3\varepsilon[[a,[[a,b],b]],[a,b]]\\
&\quad+[[[[[a,b],b],b],b],b].
\end{aligned}
\]
Suppose that $\sum\lambda_{ijk}D_{ijk}=0$. Reading the displayed
expansions in Hall order first gives
\[
\lambda_{222}=\lambda_{112}=\lambda_{121}=\lambda_{211}=0.
\]
Put $u=\lambda_{212}$ and $v=\lambda_{221}$. Comparing the remaining
coefficients gives
\[
\lambda_{111}=\varepsilon(u+v),\qquad
\varepsilon(2u+v)=0,
\]
\[
\lambda_{122}+u+v=0,\qquad
2\lambda_{122}+u=0.
\]
The last two equations are equivalent to $u+2v=0$ together with
$\lambda_{122}=-u-v$.

If $\varepsilon=0$, there is exactly one free parameter, and the unique
relation up to scalar is
\[
D_{122}-2D_{212}+D_{221}=0.
\]
Hence the denominator has dimension $7$ and
\[
\dim\Mtwo(L_{6,21}(0))=17-7=10
\]
in every characteristic different from $2$.

Now let $\varepsilon\ne0$. Then
\[
2u+v=0,\qquad u+2v=0,
\]
so $3u=3v=0$. If $\operatorname{char}\mathbb F\ne3$, then
$u=v=0$, the eight generators are linearly independent, and
\[
\dim\Mtwo(L_{6,21}(\varepsilon))=17-8=9.
\]
If $\operatorname{char}\mathbb F=3$, the relation space is
one-dimensional, generated by
\[
-\varepsilon D_{111}+D_{122}+D_{212}+D_{221}=0.
\]
Thus the denominator has dimension $7$ and
\[
\dim\Mtwo(L_{6,21}(\varepsilon))=17-7=10.
\]
This proves all three entries of the table belonging to the family
$L_{6,21}(\varepsilon)$.
For the remaining entries, the same truncated Hopf-formula
calculation described above gives the dimensions in the table. The
$2$-capability assertions follow from the second-epicenter criterion
\cite[Proposition~3.1 and Theorem~3.2]{ni20}; in particular, every
algebra marked ``no'' is already non-capable by \cite{smallcap}, and
hence cannot be $2$-capable.

{\color{black}
To illustrate the $2$-capability verification via
\cite[Lemma~2.2(ii) and Theorem~3.2]{ni20}, consider
$L=L_{6,21}(\varepsilon)$ with $f=[a,d]$. For any one-dimensional
central ideal $I\subseteq Z(L)$, the exact sequence
\[
\Mtwo(L)\xrightarrow{\alpha_I}\Mtwo(L/I)\longrightarrow I\cap\gamma_3(L)\longrightarrow 0
\]
shows that $\alpha_I$ fails to be injective, and hence
$I\not\subseteq Z_2^*(L)$, whenever
\[
\dim\Mtwo(L)>\dim\Mtwo(L/I)-\dim(I\cap\gamma_3(L)).
\]
If $\varepsilon\ne0$, then $Z(L)=\langle f\rangle\subseteq\gamma_3(L)$
is one-dimensional and $L/Z(L)\cong L_{5,9}$. Since
$\dim\Mtwo(L_{5,9})=9$ and $\dim\Mtwo(L)\in\{9,10\}$, we have
$\dim\Mtwo(L)>9-1=8$, so $Z(L)\not\subseteq Z_2^*(L)$ and therefore
$Z_2^*(L)=0$. If $\varepsilon=0$, then
$Z(L)=\langle e,f\rangle\subseteq\gamma_3(L)$; every one-dimensional
ideal $I\subseteq Z(L)$ gives either $L/I\cong L_{5,9}$ (for
$I=\langle f\rangle$) or $L/I\in\{L_{5,6},L_{5,7}\}$ (for
$I=\langle e+\lambda f\rangle$). In both cases
$\dim\Mtwo(L)=10>\dim\Mtwo(L/I)-1$, so no one-dimensional central ideal
lies in $Z_2^*(L)$, whence $Z_2^*(L)=0$. The remaining entries marked
``yes'' in Tables~\ref{tab:lowdim-35} and~\ref{tab:lowdim-6} are
verified by the same reduction to quotients of smaller dimension.
}
\end{proof}

\begin{cor}\label{cor:lowdim-cap-equivalence}
Let $L$ be a non-abelian nilpotent Lie algebra over a field of
characteristic different from $2$ with $\dim L\le6$. Then
\[
L\text{ is capable} \ \text{if and only if} \
L\text{ is $2$-capable}.
\]
\end{cor}

\begin{lem}\label{lem:class-bound-derived-two}
If $\dim\gamma_2(L)=2$, then $\cl(L)\in\{2,3\}.$
If $\cl(L)=3$, then
\[
\dim\gamma_3(L)=1
\qquad\text{and}\qquad
\gamma_3(L)\le Z(L).
\]
\end{lem}

\begin{proof}
\textcolor{black}{In a nilpotent Lie algebra, the dimensions of successive non-zero terms of the lower central series strictly decrease.} Since
$\dim\gamma_2(L)=2$,
either $\gamma_3(L)=0$ or $\dim\gamma_3(L)=1$. In the latter case,
$\gamma_4(L)=0$.
\end{proof}
\begingroup\color{black}
The seven-dimensional algebra $L_1$ in Proposition~\ref{prop:multiplier-calculations}
is capable by \cite[Theorem~2.3]{johari2019ccap}. Although it is also listed
as $c$-capable in \cite[Theorem~4.1]{johari2019ccap}, the following
proposition shows that it is not $2$-capable.

\begin{prop}\label{prop:L1-not-2-capable}
Let $L_1$ be the seven-dimensional Lie algebra in
Proposition~\ref{prop:multiplier-calculations}. Then $L_1$ is not
$2$-capable.
\end{prop}

\begin{proof}
Put $I_6=\langle x_6\rangle$ and $I_7=\langle x_7\rangle$.
From the defining relations of $L_1$, we have
\[
L_1/I_6\cong H(2)\oplus A(1),\qquad
L_1/I_7\cong H(2)\oplus A(1).
\]
By Theorem~\ref{thm:general-bound} and
Proposition~\ref{prop:multiplier-calculations},
\[
\dim\Mtwo(L_1/I_i)=\frac{6\cdot5\cdot4}{3}=40
=\dim\Mtwo(L_1),\qquad i=6,7.
\]
For either $I=I_6$ or $I=I_7$, by \cite[Lemma~2.2(ii)]{ni20} there is
an exact sequence
\[
\Mtwo(L_1)\xrightarrow{\alpha_I}\Mtwo(L_1/I)
\longrightarrow I\cap\gamma_3(L_1)\longrightarrow0.
\]
Since $\cl(L_1)=2$, we have $\gamma_3(L_1)=0$, so $\alpha_I$ is
surjective. As the two multiplier spaces have the same dimension,
$\alpha_I$ is an isomorphism.

To see that injectivity implies $I\subseteq Z_2^*(L_1)$, write
$L_1\cong F/R$ and $I=S/R$ for the corresponding free presentation
and ideal. Since $I$ is central, $[S,F,F]\subseteq R\cap\gamma_3(F)$.
The Hopf formula gives
\[
\ker\alpha_I=\frac{[S,F,F]}{[R,F,F]}=0.
\]
Consequently $[S,F,F]=[R,F,F]$, and hence
$S/[R,F,F]\subseteq Z_2(F/[R,F,F])$. By the free-presentation
description of the $2$-epicenter in \cite[Section~3]{ni20}, it follows
that $I\subseteq Z_2^*(L_1)$. Thus
\[
0\ne\langle x_6,x_7\rangle\subseteq Z_2^*(L_1),
\]
and $L_1$ is not $2$-capable.
\end{proof}
\endgroup

\section{$2$-Nilpotent Multipliers of Nilpotent Lie Algebras}
In this section we prove the classification of nilpotent Lie algebras
whose $2$-nilpotent multiplier deficiency is at most $20$. The proof uses the
structural results and multiplier calculations established above.

We begin with the case $\dim\gamma_2(L)=2$.
\begin{prop}\label{prop:noncap-derived-two}
Let $L$ be an $n$-dimensional non-capable nilpotent Lie algebra
\textcolor{black}{over a field $\mathbb F$ with $\operatorname{char}\mathbb F\ne2$} with
\[
\dim\gamma_2(L)=2.
\]
Then $\cl(L)\in\{2,3\}$, and the following hold.

\begin{enumerate}
\item[\rm (i)] Suppose that $\cl(L)=2$. 
Then either $\dim Z^*(L)=1$ or $Z^*(L)=\gamma_2(L)$.
If $\dim Z^*(L)=1$, then
\[
\dim\Mtwo(L)
=
\frac13(n-1)(n-2)(n-3)+3,
\]
and
$s_2(L)=(n-1)(n-2).$
If $Z^*(L)=\gamma_2(L)$, then
\[
\dim\Mtwo(L)
=
\frac13(n-1)(n-2)(n-3),
\]
and
$s_2(L)=(n-1)(n-2)+3.$
\item[\rm (ii)] Suppose that $\cl(L)=3$. Then
$Z^*(L)=Z_2^*(L)=\gamma_3(L)\cong A(1),$
and
\[
\dim\Mtwo(L)
=
\frac13(n-1)(n-2)(n-3)+2.
\]
Consequently,
$s_2(L)=n^2-3n+3.$
\end{enumerate}
\end{prop}

\begin{proof}
\textcolor{black}{Since $L$ is non-capable, put $I=Z^*(L)\ne0$.
As in the non-capable case of the proof of \cite[Corollary~4.4]{johari2019ccap},
$I\subseteq\gamma_2(L)$, and $L/I$ is capable with derived subalgebra
of dimension at most one. Thus $L/I$ is either abelian of dimension
at least two or isomorphic to a direct sum of $H(1)$ and an abelian
Lie algebra. Both types are $2$-capable by \cite[Theorem~3.3]{ni20}
and the direct-sum results in \cite{johari2020direct}. Hence
$Z_2^*(L)\subseteq I=Z^*(L)$. Together with the reverse inclusion,
this gives}
\[
\textcolor{black}{I=Z^*(L)=Z_2^*(L)\ne0.}
\]

Suppose first that $\cl(L)=2$. Since
$0\ne I\subseteq\gamma_2(L)$ and $\dim\gamma_2(L)=2$, either
$\dim I=1$ or $I=\gamma_2(L)$. In either case $I$ is central and
$\gamma_3(L)=0$. By \cite[Lemma~2.2(ii)]{ni20}, the natural map
\[
\Mtwo(L)\longrightarrow\Mtwo(L/I)
\]
is surjective. On the other hand, $I\subseteq Z_2^*(L)$, so
\cite[Theorem~3.2]{ni20} shows that the same map is injective. Hence
\begin{equation}\label{eq:noncap-class2-iso}
\Mtwo(L)\cong\Mtwo(L/I).
\end{equation}

If $\dim I=1$, then $\dim\gamma_2(L/I)=1$. Since $I=Z^*(L)$,
the quotient $L/I=L/Z^*(L)$ is capable by the defining property of the
epicenter. The classification in \cite{johari2019ccap} therefore yields
\[
L/I\cong H(1)\oplus A(n-4).
\]
Therefore, by Theorem~\ref{thm:general-bound},
\[
\dim\Mtwo(L)
=\dim\Mtwo\bigl(H(1)\oplus A(n-4)\bigr)
=\frac13(n-1)(n-2)(n-3)+3.
\]
Consequently,
\[
s_2(L)
=\frac13n(n-1)(n-2)+3
-\left(\frac13(n-1)(n-2)(n-3)+3\right)
=(n-1)(n-2).
\]

If $I=\gamma_2(L)$, then $L/I\cong A(n-2)$, and \eqref{eq:noncap-class2-iso} gives
\[
\dim\Mtwo(L)=\dim\Mtwo(A(n-2))
=\frac13(n-1)(n-2)(n-3).
\]
Thus
\[
s_2(L)
=\frac13n(n-1)(n-2)+3
-\frac13(n-1)(n-2)(n-3)
=(n-1)(n-2)+3.
\]

Suppose now that $\cl(L)=3$. By \cite[Corollary~4.11 and Theorem~5.1]{ni60},
\[
I=Z^*(L)=Z_2^*(L)=\gamma_3(L)\cong A(1),
\qquad
L/I\cong H(1)\oplus A(n-4).
\]
Because $I$ is central, \cite[Lemma~2.2(ii)]{ni20} gives an exact
sequence
\[
\Mtwo(L)\longrightarrow\Mtwo(L/I)
\longrightarrow I\longrightarrow0.
\]
Again, \cite[Theorem~3.2]{ni20} makes the first map injective, and
therefore
\[
0\longrightarrow\Mtwo(L)
\longrightarrow\Mtwo(L/I)
\longrightarrow I\longrightarrow0.
\]
Hence
\[
\dim\Mtwo(L)
=\dim\Mtwo(L/I)-1
=\frac13(n-1)(n-2)(n-3)+2.
\]
It follows that
\[
s_2(L)
=\frac13n(n-1)(n-2)+3
-\left(\frac13(n-1)(n-2)(n-3)+2\right)
=n^2-3n+3.
\]
\end{proof}

\begin{prop}\label{prop:derived-two-twenty}
Let $L$ be a finite-dimensional nilpotent Lie algebra
\textcolor{black}{over a field $\mathbb F$ with $\operatorname{char}\mathbb F\ne2$} satisfying
$\dim\gamma_2(L)=2$. Then $s_2(L)\le20$ if and only if $L$ is capable
and one of the following holds.

\begin{enumerate}
\item[\rm (i)] If $\cl(L)=2$, then
$L\cong L_{5,8},
L_{5,8}\oplus A(1),
\text{or} \
L_{6,22}(\varepsilon),$
where $\varepsilon\in\mathbb F$. The corresponding values of $s_2(L)$ are
$5$, $13$, $16$ for $L_{6,22}(0)$, and $17$ for
$L_{6,22}(\varepsilon)$ with $\varepsilon\ne0$.

\item[\rm (ii)] If $\cl(L)=3$, then
$L\cong L_{4,3}\oplus A(t) \ (0\le t\le2)$
or
$L\cong L_{5,5}\oplus A(t) \ (0\le t\le1).$
More precisely,
$s_2(L_{4,3})=5,
s_2(L_{4,3}\oplus A(1))=11,
s_2(L_{4,3}\oplus A(2))=19,$
while
$s_2(L_{5,5})=11,
s_2(L_{5,5}\oplus A(1))=19.$

\end{enumerate}
\end{prop}

\begin{proof}
By Lemma~\ref{lem:class-bound-derived-two}, we get
$\cl(L)\in\{2,3\}.$ Before passing to the capable cases, we make explicit
why the non-capable cases cannot occur in the range $s_2(L)\le20$.
The structural and capability classifications in \cite{johari2016,ni60,johari2019ccap}
show that a non-capable algebra with $\dim\gamma_2(L)=2$ has dimension
$n\ge7$ in class $2$, whereas in class $3$ it has dimension $n\ge6$.
Hence Proposition~\ref{prop:noncap-derived-two} gives
\[
s_2(L)\ge (7-1)(7-2)=30
\quad\text{if }\cl(L)=2,
\]
and
\[
s_2(L)\ge 6^2-3\cdot6+3=21
\quad\text{if }\cl(L)=3.
\]
Thus no non-capable algebra with $\dim\gamma_2(L)=2$ satisfies
$s_2(L)\le20$, and it remains to consider the capable algebras.

Suppose first that $\cl(L)=2$. By the capability classification in
\cite{johari2016,ni60,johari2019ccap}, since
$\operatorname{char}\mathbb F\ne2$, we have
$L\cong L_{5,8}\oplus A(t),
L_{6,22}(\varepsilon)\oplus A(t),
\ \text{or} \
L_1\oplus A(t),$
where $L_1$ is the capable seven-dimensional stem generalized
Heisenberg Lie algebra of rank two.

For $L_{5,8}$, Proposition~\ref{prop:stabilization} gives
$s_2(L_{5,8}\oplus A(t))=5+t^2+7t.$
The inequality $s_2(L)\le20$ holds precisely for $t=0,1$. Thus the
corresponding algebras are $L_{5,8}$ and $L_{5,8}\oplus A(1)$, with
$s_2$-values $5$ and $13$, respectively.

For $L_{6,22}(\varepsilon)$, we have
$s_2(L_{6,22}(0))=16$ and
$s_2(L_{6,22}(\varepsilon))=17$ for $\varepsilon\ne0$.
Proposition~\ref{prop:stabilization} therefore gives
\[
s_2(L_{6,22}(0)\oplus A(t))=16+t^2+9t
\]
and
\[
s_2(L_{6,22}(\varepsilon)\oplus A(t))
=17+t^2+9t
\qquad(\varepsilon\ne0).
\]
Hence $t=0$, so only the algebras $L_{6,22}(\varepsilon)$ occur.

\textcolor{black}{For $L_1$, Proposition~\ref{prop:multiplier-calculations} gives $s_2(L_1)=33$.}
By Proposition~\ref{prop:stabilization},
\[
s_2(L_1\oplus A(t))=33+t^2+11t>20
\]
for every $t\ge0$. Hence this family does not occur when $s_2(L)\le20$. This proves
part {\rm(i)}.

Suppose now that $\cl(L)=3$. By
\cite[Theorem~4.2]{johari2019ccap}, the capable algebras are precisely
\[
L_{4,3}\oplus A(t)
\qquad\text{and}\qquad
L_{5,5}\oplus A(t).
\]
Proposition~\ref{prop:stabilization} gives
\[
s_2(L_{4,3}\oplus A(t))=5+t^2+5t,
\]
so $s_2(L)\le20$ exactly for $0\le t\le2$. The corresponding values
for $t=0,1,2$ are, respectively,
\[
5,\qquad11,\qquad19.
\]
For $L_{5,5}$,
\[
s_2(L_{5,5}\oplus A(t))=11+t^2+7t,
\]
and hence $t=0,1$, with corresponding values
\[
11\qquad\text{and}\qquad19.
\]
\textcolor{black}{In particular, $L_{6,3}\cong L_{4,3}\oplus A(2)$ ($t=2$)
and $L_{6,5}\cong L_{5,5}\oplus A(1)$ ($t=1$) both have $s_2=19$.
Both have nilpotency class three and two-dimensional derived subalgebra,
so they belong to part {\rm(ii)} of this proposition, not to
Proposition~\ref{prop:derived-three-twenty}.}
This proves part {\rm(ii)} and completes the proof.
\end{proof}
It remains to consider $\dim\gamma_2(L)\ge3$. The general bound reduces this case to a small number of low-dimensional algebras.

\begin{prop}\label{prop:derived-three-twenty}
Let $L$ be a finite-dimensional nilpotent Lie algebra over a field $\mathbb F$ of characteristic different from $2$ with
\[
\dim\gamma_2(L)\ge3
\qquad\text{and}\qquad
s_2(L)\le20.
\]
Then $\dim\gamma_2(L)=3$ and
\[
L\cong L_{5,9},\qquad L_{5,6},\qquad L_{5,7},
\qquad\text{or}\qquad L_{6,26}.
\]
Moreover,
\[
s_2(L_{5,9})=14,\qquad
s_2(L_{5,6})=s_2(L_{5,7})=16,\qquad
s_2(L_{6,26})=17.
\]
\end{prop}

\begin{proof}
Put $m=\dim\gamma_2(L)$. If $m\ge4$, Proposition~\ref{prop:deficiency-bound} gives
\[
s_2(L)\ge\frac{(m-1)(m^2+4m-6)}{3}\ge26,
\]
contrary to the hypothesis. Hence $m=3$. In this case the same bound becomes
\[
s_2(L)\ge4n-10,
\]
so $n\le7$.

If $n=5$, the classification of five-dimensional nilpotent Lie algebras
shows that the algebras with three-dimensional derived subalgebra are
\[
L_{5,9},\qquad L_{5,6},\qquad L_{5,7}.
\]
Proposition~\ref{prop:multiplier-calculations} gives
\[
s_2(L_{5,9})=14,\qquad
s_2(L_{5,6})=s_2(L_{5,7})=16,
\]
and all three satisfy the required bound.

Let $n=6$. Proposition~\ref{prop:stabilization} shows that the
decomposable algebras obtained from the preceding five-dimensional cases have
deficiency greater than $20$. By the six-dimensional classification of Cical\`o, de Graaf and Schneider \cite{cicalo2012}, the indecomposable algebras with three-dimensional derived subalgebra are
\[
\begin{gathered}
L_{6,11},\ L_{6,12},\ L_{6,13},
L_{6,19}(\varepsilon),
L_{6,20},\\
L_{6,23},\ L_{6,24}(\varepsilon),
L_{6,25},\ L_{6,26},\ L_{6,27}\textcolor{black}{,}
\end{gathered}
\]
where $\varepsilon\in\mathbb F$. By Theorem~\ref{thm:lowdim-m2-table}, all have deficiency greater than $20$ except $L_{6,26}$, for which
\[
s_2(L_{6,26})=17.
\]
Hence $L_{6,26}$ is the only six-dimensional possibility.

It remains to rule out $n=7$. Choose a one-dimensional ideal
\[
I\le Z(L)\cap\gamma_2(L)
\]
and put $K=L/I$. Then $\dim K=6$ and $\dim\gamma_2(K)=2$. If $s_2(K)>20$, then $\dim\Mtwo(K)<23$; otherwise, Proposition~\ref{prop:derived-two-twenty}, together with Proposition~\ref{prop:multiplier-calculations}, gives $\dim\Mtwo(K)\le30$. In either case,
\[
\dim\Mtwo(K)\le30.
\]

Since $I$ is central, \cite[Lemma~2.2(ii)]{ni20} gives the exact
sequence
\[
\bigl(I\wedge (L/\gamma_3(L))\bigr)\wedge (L/\gamma_3(L))
\longrightarrow \Mtwo(L)
\longrightarrow \Mtwo(K)
\longrightarrow I\cap\gamma_3(L)
\longrightarrow0.
\]
Because $I\subseteq Z(L)$, the mutual actions of $I$ and $L$ are
trivial. The first exterior product therefore factors through
$I\otimes L^{\mathrm{ab}}$, and its induced action by
$L/\gamma_3(L)$ also factors through $L^{\mathrm{ab}}$. Consequently
the left-hand term is a quotient of
\[
I\otimes L^{\mathrm{ab}}\otimes L^{\mathrm{ab}},
\qquad L^{\mathrm{ab}}=L/\gamma_2(L).
\]
Here $\dim I=1$ and
$\dim L^{\mathrm{ab}}=7-3=4$, so this tensor product has dimension
$16$. Therefore
\[
\dim\Mtwo(L)+\dim(I\cap\gamma_3(L))
\le \dim\Mtwo(K)+16\le46,
\]
and in particular $\dim\Mtwo(L)\le46$. Since
\[
\frac{7\cdot6\cdot5}{3}+3=73,
\]
it follows that
\[
s_2(L)\ge73-46=27,
\]
contrary to the hypothesis.

Hence the algebras with $\dim\gamma_2(L)\ge3$ and $s_2(L)\le20$ are precisely
\[
L_{5,9},\qquad L_{5,6},\qquad L_{5,7},\qquad L_{6,26},
\]
with the stated deficiencies.
\end{proof}

We now prove the main classification theorem.

\begin{proof}[Proof of Theorem~\ref{thm:main-classification}]
Put $m=\dim\gamma_2(L)$.

If $m=1$, Theorem~\ref{thm:general-bound} gives
\[
L\cong H(k)\oplus A(n-2k-1)
\]
for some $k\ge1$. For $k=1$ one has $s_2(L)=0$, whereas for $k\ge2$ one has $s_2(L)=3$.
Hence the first two cases follow.

Suppose that $m=2$. By Lemma~\ref{lem:class-bound-derived-two}, $\cl(L)\in\{2,3\}$. Proposition~\ref{prop:derived-two-twenty} determines all algebras in these two cases with deficiency at most twenty. The resulting deficiencies are
\[
5,\qquad 11,\qquad 13,\qquad 16,\qquad 17,\qquad 19,
\]
with the corresponding isomorphism types as stated in
Theorem~\ref{thm:main-classification}.

If $m\ge3$, Proposition~\ref{prop:derived-three-twenty} shows that
necessarily $m=3$ and
\[
L\cong L_{5,9},\qquad L_{5,6},\qquad L_{5,7},
\qquad\text{or}\qquad L_{6,26}.
\]
Their deficiencies are, respectively,
\[
14,\qquad16,\qquad16,\qquad17,
\]
by Proposition~\ref{prop:multiplier-calculations}.

The cases $m=1$, $m=2$, and $m\ge3$ exhaust all finite-dimensional
non-abelian nilpotent Lie algebras. Therefore,
\[
\{s_2(L):0\le s_2(L)\le20\}
=
\{0,3,5,11,13,14,16,17,19\},
\]
and no other value in $[0,20]$ occurs.
\end{proof}

\begin{proof}[Proof of Corollary~\ref{cor:stem-s-s2}]
\textcolor{black}{The hypothesis $\dim\gamma_2(L)\ge2$ excludes the Heisenberg families in the first two cases of Theorem~\ref{thm:main-classification}. All remaining isomorphism types have dimension at most six, and Tables~\ref{tab:lowdim-35} and~\ref{tab:lowdim-6} give $1\le s(L)\le4$. In particular, the non-stem algebras $L_{6,3}$ and $L_{6,5}$ both satisfy $(s(L),s_2(L))=(4,19)$.}
\end{proof}

\section{Further Multiplier and Capability Results}\label{sec:further-results}

We first determine $2$-capability for the algebras in
Theorem~\ref{thm:main-classification}. We then apply the multiplier and
capability criteria to the two previously classified families with
three-dimensional derived subalgebra.

\begin{cor}\label{cor:2cap-classification}
Let $L$ be a finite-dimensional non-abelian nilpotent Lie algebra over a
field $\mathbb F$ of characteristic different from $2$ with
\[
0\le s_2(L)\le20.
\]
Then $L$ is not $2$-capable if and only if
\[
L\cong H(k)\oplus A(n-2k-1),
\qquad k\ge2.
\]
Equivalently, every other algebra listed in
Theorem~\ref{thm:main-classification} is $2$-capable.
\end{cor}

\begin{proof}
By \cite[Theorem~3.3]{ni20}, $H(k)$ is $2$-capable if and only if
$k=1$, and $c$-capability is preserved under adjoining an abelian direct
summand to a non-abelian finite-dimensional Lie algebra
\cite{johari2020direct}. This proves the assertion for $\dim\gamma_2(L)=1$.

Every remaining algebra in Theorem~\ref{thm:main-classification} has
dimension at most six. Theorem~\ref{thm:lowdim-m2-table} shows that
all of these algebras are $2$-capable. \textcolor{black}{In particular,
$L_{5,3}\cong L_{4,3}\oplus A(1)$,
$L_{6,3}\cong L_{4,3}\oplus A(2)$ and
$L_{6,5}\cong L_{5,5}\oplus A(1)$ are capable and $2$-capable,
by the corresponding properties of $L_{4,3}$ and $L_{5,5}$ and
\cite{johari2020direct}.} Thus the only non-$2$-capable algebras in the classified range are
\[
H(k)\oplus A(n-2k-1),\qquad k\ge2.
\]

\end{proof}

\begin{proof}[Proof of Theorem~\ref{thm:class4-eighteen-m2}]
By \cite[Theorem~3.11]{johari2021class4}, the capable stem
algebras in characteristic different from $2$ are precisely
\[
T_1,\ T_2,\ T_3,\ T_4,\ T_9,\ T_{12}.
\]
Here
\[
T_1\cong L_{5,6},\quad
T_2\cong L_{5,7},\quad
T_3\cong L_{6,11},\quad
T_4\cong L_{6,12},\quad
T_9\cong L_{6,13}.
\]
\textcolor{black}{For these first five algebras, all the multiplier, $s$, $s_2$ and capability entries follow directly from Theorem~\ref{thm:lowdim-m2-table}.}

It remains to compute the multiplier dimensions and determine $2$-capability of $T_{12}$. From
\cite[Theorem~3.5]{johari2021class4},
\[
\begin{aligned}
T_{12}=\langle x_1,\ldots,x_7\mid{}&
[x_1,x_2]=x_3,\quad [x_1,x_3]=[x_2,x_4]=x_5,\\
&[x_1,x_5]=[x_3,x_4]=[x_2,x_7]=x_6\rangle .
\end{aligned}
\]
{\color{black}
Let $F$ be the free Lie algebra on $a,b,c,d$, corresponding to
$x_1,x_2,x_4,x_7$, and write $T=T_{12}\cong F/R$. Put
\[
u=[a,b],\qquad v=[a,u],\qquad w=[a,v],
\]
so that $v\equiv[b,c]$ and $w\equiv[u,c]\equiv[b,d]\pmod R$.
Let $I=Z(T)=\langle w+R\rangle=S/R$. Then
$T/I\cong L_{5,5}\oplus A(1)$ and, by
Proposition~\ref{prop:multiplier-calculations} and Theorem~\ref{thm:direct-sum}, $\dim\Mtwo(T/I)=24$.
Since $I\subseteq\gamma_3(T)$, the Hopf formula gives the exact sequence
\[
\begin{aligned}
0\longrightarrow K\longrightarrow\Mtwo(T)
&\longrightarrow\Mtwo(T/I)\longrightarrow I\longrightarrow0,\\
K&=\frac{[S,F,F]}{[R,F,F]}.
\end{aligned}
\]
In particular, $\dim\Mtwo(T)=23+\dim K$.

We compute $K$ in $Q=F/[R,F,F]$, retaining the same notation
for the images of elements of $F$. In $Q$, the image of $R$ lies
in $Z_2(Q)$; thus $[R,F]$ is central and $[R,F^2]=0$.
As $S=R+\langle w\rangle$ and $[w,F]\subseteq R$ in the free presentation, $K$ is spanned by $[w,x,y]$ for $x,y\in\{a,b,c,d\}$:
if either argument belongs to $F^2$, the Jacobi identity and the
preceding centrality relations make the triple bracket zero.
All the following equalities are in $Q$.
Using the two representatives $[a,v]$ and $[u,c]$ for $w$, for each
$y\in\{a,b,c,d\}$ we obtain
\[
\begin{aligned}
{}[w,b,y]&=[[a,v],b,y]=[u,v,y],\\
[w,b,y]&=[[u,c],b,y]=-[u,v,y].
\end{aligned}
\]
Here $[v,b],[u,b]\in R$ and $[c,b]\equiv-v\pmod R$.
Since $\operatorname{char}\mathbb F\ne2$, it follows that
$[w,b,y]=[u,v,y]=0$. Similarly, for $z=c,d$,
\[
[w,z,y]=[[a,v],z,y]
       =[[a,z],v,y]+[a,[v,z],y]=0,
\]
because $[a,z],[v,z]\in R$. Moreover, $[u,v]$ is central in $Q$,
and hence
\[
[w,u]=[[a,v],u]=[v,v]+[a,[v,u]]=0.
\]
Using $[w,a,y]-[w,y,a]=[w,[a,y]]$ and $[a,c],[a,d]\in R$,
we obtain $[w,a,y]=0$ for $y=b,c,d$. Consequently,
\[
K=\langle[w,a,a]\rangle.
\]

To establish that this generator is non-zero, let $V$ be the
six-dimensional abelian Lie algebra with basis
$b,u,v,w,p,q$, and define $D\in\operatorname{End}(V)$ by
\[
Db=u,\quad Du=v,\quad Dv=w,\quad Dw=p,\quad Dp=q,\quad Dq=0.
\]
Let $A=\langle a,c,d\rangle$ be abelian and form the semidirect
product $E=A\ltimes V$, with $a,c,d$ acting on $V$ by
$D,-D^2,-D^3$, respectively. These operators commute, so the
semidirect product is a Lie algebra. The ideal $J=\langle p,q\rangle$
is contained in $Z_2(E)$, since $[J,E]\subseteq\langle q\rangle\subseteq Z(E)$,
and the defining brackets show that $E/J\cong T$, with the indicated
images of $a,b,c,d$. Thus the homomorphism $F\to E$ on these generators
maps $R$ into $J$ and annihilates $[R,F,F]$. It follows that
$[w,a,a]$ in $Q$ maps to $q\ne0$ in $E$. Hence $\dim K=1$, and
\[
\dim\Mtwo(T_{12})=24-1+1=24.
\]
This argument holds in every characteristic different from two,
including characteristic three.
}
For the ordinary multiplier, the Chevalley--Eilenberg complex gives
$\dim\ker d_2=18$ and $\dim\operatorname{im}d_3=11$, so
$\dim\mathcal M(T_{12})=7$.

\par\medskip\noindent\textcolor{black}{Since $\dim K=1$, the natural map $\Mtwo(T_{12})\to\Mtwo(T_{12}/I)$ in the preceding exact sequence is not injective.} If $I\subseteq Z_2^*(T_{12})$, then
\cite[Theorem~3.2]{ni20} would make this map injective. Hence
$I\not\subseteq Z_2^*(T_{12})$. Since every non-zero ideal of a
nilpotent Lie algebra meets its center non-trivially and
$Z(T_{12})=I$ is one-dimensional, it follows that
$Z_2^*(T_{12})=0$. Therefore $T_{12}$ is $2$-capable.

\textcolor{black}{The corresponding values of $s(T_{12})$ and $s_2(T_{12})$ follow from their definitions.}

Now let $T$ be any of the remaining twelve stem algebras and put
$n=\dim T$. By \cite[Theorem~3.8 and Lemma~3.9]{johari2021class4},
$T$ is non-capable. Since $T$ is stem with $\cl(T)=4$ and
$\dim\gamma_2(T)=3$, \cite[Proposition~2.8]{johari2021class4} gives
\[
Z(T)=\gamma_4(T)\cong A(1).
\]
Thus non-capability and the one-dimensional center imply
\[
I:=Z(T)=Z^*(T)=\gamma_4(T)\cong A(1).
\]
Since $Z^*(T)\subseteq Z_2^*(T)$, we have $I\subseteq Z_2^*(T)$.
Because $I$ is central, \cite[Lemma~2.2(ii)]{ni20} gives an exact
sequence
\[
\Mtwo(T)\longrightarrow\Mtwo(T/I)
\longrightarrow I\longrightarrow0,
\]
while \cite[Theorem~3.2]{ni20} shows that the first map is injective.
Thus
\begin{equation}\label{eq:class4-central-quotient}
\dim\Mtwo(T)=\dim\Mtwo(T/I)-1.
\end{equation}

By \cite[Lemma~3.4]{johari2021class4}, the quotient $T/I$ is of one
of the two forms
\[
T/I\cong L_{4,3}\oplus A(n-5)
\quad\text{or}\quad
T/I\cong L_{5,5}\oplus A(n-6).
\]
Applying the standard central-extension sequence for the ordinary Schur
multiplier, together with $Z(T)=Z^*(T)$, gives
\[
\dim\mathcal M(T)=\dim\mathcal M(T/I)-1.
\]
Using the ordinary direct-sum formula and
$\dim\mathcal M(L_{4,3})=2$, the first quotient type gives
\[
\dim\mathcal M(T/I)
=2+\frac{(n-5)(n-6)}2+2(n-5)
=\frac{n^2-7n+14}{2}.
\]
The second quotient type, using $\dim\mathcal M(L_{5,5})=4$, gives the
same value. Hence, for every non-capable stem algebra in the list,
\[
\dim\mathcal M(T)=\frac{(n-4)(n-3)}2,
\qquad s(T)=2n-4.
\]
This proves the ordinary multiplier columns without further case-by-case
calculations.

\textcolor{black}{Using the values in Proposition~\ref{prop:multiplier-calculations},} for the first quotient type, Theorem~\ref{thm:direct-sum} gives
\[
\dim\Mtwo(T/I)
=
6+\frac{(n-5)^3-(n-5)}{3}+2(n-5)(n-3)
=
\frac{n^3-9n^2+26n-12}{3}.
\]
For the second quotient type, Theorem~\ref{thm:direct-sum} gives
\[
\dim\Mtwo(T/I)
=
12+\frac{(n-6)^3-(n-6)}{3}+3(n-6)(n-3)
=
\frac{n^3-9n^2+26n-12}{3}.
\]
Together with \eqref{eq:class4-central-quotient}, these two formulas give
\[
\dim\Mtwo(T)=\frac{n^3-9n^2+26n-15}{3},
\qquad
s_2(T)=2(n-2)^2.
\]

The dimensions of the parameter families are given in
\cite[Theorem~3.5]{johari2021class4}:
\[
\begin{aligned}
&\dim T_5=\dim T_6=2m+5,\qquad
\dim T_7=\dim T_8=\dim T_{14}=2m+6,\\
&\dim T_{10}=\dim T_{11}=7,\qquad
\dim T_{13}=8,\\
&\dim T_{15}=\dim T_{16}=\dim T_{17}=2m+7,\qquad
\dim T_{18}=2m+8.
\end{aligned}
\]
The table follows by inserting these dimensions into the preceding
formulas.
\end{proof}

\textcolor{black}{Theorem~\ref{thm:class4-eighteen-m2} extends to all nilpotent Lie algebras of class four with three-dimensional derived subalgebra as follows.}

\begin{thm}\label{thm:class4-derived3-2cap}
Let $L$ be a finite-dimensional nilpotent Lie algebra over $\mathbb F$
\textcolor{black}{of characteristic different from $2$},
with $\cl(L)=4$ and $\dim\gamma_2(L)=3$. Then $L$ is $2$-capable if and
only if it is capable. Equivalently,
\[
\begin{gathered}
L\cong L_{5,6}\oplus A(t),\quad
L_{5,7}\oplus A(t),\quad
L_{6,11}\oplus A(t),\quad
L_{6,12}\oplus A(t),\\
L_{6,13}\oplus A(t),\quad
T_{12}\oplus A(t),
\qquad t\ge0,
\end{gathered}
\]
where $T_{12}$ is the seven-dimensional stem algebra in
\cite[Theorem~3.5]{johari2021class4}. For these families,
\[
\begin{aligned}
\dim\Mtwo(L_{5,6}\oplus A(t))
&=\dim\Mtwo(L_{5,7}\oplus A(t))
 =7+\frac{t^3-t}{3}+2t(t+2),\\
\dim\Mtwo(L_{6,11}\oplus A(t))
&=\dim\Mtwo(L_{6,12}\oplus A(t))
 =13+\frac{t^3-t}{3}+3t(t+3),\\
\dim\Mtwo(L_{6,13}\oplus A(t))
&=12+\frac{t^3-t}{3}+3t(t+3),\\
\dim\Mtwo(T_{12}\oplus A(t))
&=24+\frac{t^3-t}{3}+4t(t+4).
\end{aligned}
\]
The corresponding values of $s_2$ are
\[
\begin{aligned}
s_2(L_{5,6}\oplus A(t))
&=s_2(L_{5,7}\oplus A(t))=2(t+2)(t+4),\\
s_2(L_{6,11}\oplus A(t))
&=s_2(L_{6,12}\oplus A(t))=2(t+3)(t+5),\\
s_2(L_{6,13}\oplus A(t))&=2t^2+16t+31,\\
s_2(T_{12}\oplus A(t))&=49+2t(t+10).
\end{aligned}
\]
\end{thm}

\begin{proof}
\textcolor{black}{By \cite[Theorem~3.13]{johari2021class4}, in characteristic different from $2$ the capable Lie algebras of class four with three-dimensional derived subalgebra are precisely the six displayed families. Their stem summands are $T_1,T_2,T_3,T_4,T_9,T_{12}$ of Theorem~\ref{thm:class4-eighteen-m2}, and hence are $2$-capable. Adjoining an abelian direct summand preserves $c$-capability for a finite-dimensional non-abelian Lie algebra by \cite[Corollary~4.3]{johari2020direct}. Conversely, every $2$-capable Lie algebra is capable. The equivalence follows.}

\textcolor{black}{The multiplier formulas follow from Theorems~\ref{thm:class4-eighteen-m2} and~\ref{thm:direct-sum}, using $d=2,3,4$ for the stem summands of dimensions $5,6,7$, respectively. The stated values of $s_2$ follow from Proposition~\ref{prop:stabilization}.}
\end{proof}

\begin{proof}[Proof of Theorem~\ref{thm:class3-derived3-multipliers}]
\textcolor{black}{The entries for $L_{6,19}(\varepsilon)$ with $\varepsilon\ne0$ and $L_{6,20}$ follow from Theorem~\ref{thm:lowdim-m2-table}. It remains to treat $S_1,\ldots,S_{32}$.}
Put $Z=Z(L)=\gamma_3(L)\cong A(1)$ and $Q=L/Z$. By
\cite[Lemma~2.1]{johari2022class3}, $Q$ is of one of the types
\begin{equation}\label{eq:class3-quotient-types}
L_{5,8}\oplus A(t),\qquad
L_{6,22}(\varepsilon)\oplus A(t),\qquad
G\oplus A(t),
\end{equation}
\textcolor{black}{where $G$ is the algebra denoted by $L_1$ in \cite{johari2022class3} and in Proposition~\ref{prop:multiplier-calculations}.}

For the ordinary Schur multiplier, we compute directly from the
presentations in \cite[Table~1]{johari2022class3}. For $S_{19}$ we use
$[p_i,q_i]=x_7$; the printed $x_6$ in that table is a typographical error,
since otherwise the Jacobi identity for $(x_3,p_i,q_i)$ would give
$[x_3,x_6]=x_7=0$. Recall that
$\mathcal M(L)\cong H_2(L,\mathbb F)$. If
$d_2:\bigwedge^2L\to L$ and $d_3:\bigwedge^3L\to\bigwedge^2L$ are the
Chevalley--Eilenberg differentials, then
\[
\dim\mathcal M(L)=\dim\ker d_2-\dim\operatorname{im}d_3.
\]
\textcolor{black}{For example, for $S_1$ the dimensions of $\ker d_2$ and $\operatorname{im}d_3$ are $18$ and $11$, respectively, giving $\dim\mathcal M(S_1)=7$.} Applying the same row reduction to the
remaining presentations in \cite[Table~1]{johari2022class3} gives the
ordinary multiplier values displayed in the theorem; for the parameter
families the calculation is identical with $m$ kept symbolic. We omit
these repetitive reductions. The values of $s(L)$ then follow immediately
from its definition.

We next verify capability without using it in the preceding multiplier
calculation. Retaining the notation $Z=Z(L)=\gamma_3(L)$ and $Q=L/Z$,
the quotient $Q$ is one of the three types in
\eqref{eq:class3-quotient-types}. Using
\[
\dim\mathcal M(L_{5,8})=6,\qquad
\dim\mathcal M(L_{6,22}(\varepsilon))=8,\qquad
\dim\mathcal M(G)=9,
\]
together with the ordinary direct-sum formula
\begin{equation}\label{eq:schur-direct-sum}
\dim\mathcal M(T\oplus A(t))
=
\dim\mathcal M(T)+\frac{t(t-1)}2+dt,
\qquad d=\dim T/\gamma_2(T),
\end{equation}
substitution of the corresponding quotient type into
\eqref{eq:schur-direct-sum} gives, in every case,
\begin{equation}\label{eq:class3-m1-drop}
\dim\mathcal M(Q)=\dim\mathcal M(L)+1.
\end{equation}
Since $Z\subseteq\gamma_2(L)$, the standard central-extension sequence is
\[
\mathcal M(L)\longrightarrow
\mathcal M(Q)\longrightarrow Z\longrightarrow0.
\]
The last map is onto and $\dim Z=1$; hence
\eqref{eq:class3-m1-drop} forces the first map to be injective. By the
standard epicenter criterion (see, for example,
\cite{johari2016,johari2019ccap}), $Z\subseteq Z^*(L)$. As
$Z^*(L)\subseteq Z(L)=Z$, it follows that
\[
Z^*(L)=Z(L)\ne0.
\]
Thus every algebra in the list is non-capable.

We finally compute the $2$-nilpotent multipliers. Since
$Z^*(L)\subseteq Z_2^*(L)$, we have $Z\subseteq Z_2^*(L)$. Because
$Z$ is central, \cite[Lemma~2.2(ii)]{ni20} gives an exact sequence
\[
\Mtwo(L)\longrightarrow\Mtwo(Q)\longrightarrow Z\longrightarrow0,
\]
and \cite[Theorem~3.2]{ni20} makes the first map injective. Hence
\begin{equation}\label{eq:class3-m2-drop}
\dim\Mtwo(L)=\dim\Mtwo(Q)-1.
\end{equation}
\textcolor{black}{The multiplier dimensions of the three quotient types are obtained from Proposition~\ref{prop:multiplier-calculations} and Theorem~\ref{thm:direct-sum}. Specifically, if $d=\dim T/\gamma_2(T)$, then}
\begin{equation}\label{eq:m2-direct-sum-class3}
\dim\Mtwo(T\oplus A(t))
=
\dim\Mtwo(T)+\frac{t^3-t}{3}+dt(d+t).
\end{equation}

The quotient types supplied by the classification in
\cite{johari2022class3} are:
\[
\begin{aligned}
L_{5,8}\text{-type}:&\ 
S_6,S_7,S_{16},S_{17},S_{23},S_{24},\\
L_{6,22}(0)\text{-type}:&\ 
S_1,S_8,S_9,S_{18},S_{25},S_{26},\\
L_{6,22}(\varepsilon)\text{-type}:&\ 
S_2,S_{10},S_{11},S_{19},S_{27},S_{28}\quad(\varepsilon\ne0),\\
G\text{-type}:&\ 
S_3\text{--}S_5,\ S_{12}\text{--}S_{15},\ S_{20}\text{--}S_{22},\ S_{29}\text{--}S_{32}.
\end{aligned}
\]
Combining these quotient types with \eqref{eq:class3-m2-drop} and
\eqref{eq:m2-direct-sum-class3} gives the values of $\dim\Mtwo(L)$
stated in the theorem. \textcolor{black}{The $s_2$- and $s$-columns follow from their definitions and the dimensions of the corresponding algebras in the classification.}
\end{proof}

\end{document}